\documentclass[12pt]{amsart}
\usepackage{amssymb}
\usepackage{amsmath}
\usepackage{bbm}
\usepackage{color}
\usepackage{verbatim}
\usepackage{graphicx} 

\definecolor{darkred}{rgb}{0.75,0,0.3}

\newcommand\AND{\quad\text{and}\quad}
\newcommand\ep{\varepsilon}

\newcommand\N{\mathbb N}
\newcommand\Prob{\mathsf{Pr}}

\newtheoremstyle{mythm}
  {9pt}
  {9pt}
  {\itshape}
  {0pt}
  {\bfseries}
  {}
  { }
  {
  \thmname{ #1}\thmnote{ #3}}

\newtheoremstyle{mydef}
  {9pt}
  {9pt}
  {\normalfont}
  {0pt}
  {\bfseries}
  {}
  { }
  {
  \thmname{ #1}\thmnote{ #3}}

\theoremstyle{mythm}
\newtheorem{theorem1}[equation]{Theorem 1.}
\newtheorem{theorem2}[equation]{Theorem 2.}
\newtheorem{pro}[equation]{Proposition.}
\newtheorem{lem}[equation]{Lemma.}

\theoremstyle{mydef}
\newtheorem{rmk}[equation]{Remark.}

\begin{document}$\,$ \vspace{-1truecm}
\title{Some old and basic facts about random walks on groups}
\author{\bf Wolfgang Woess}

\address{\parbox{.8\linewidth}{Institut f\"ur Diskrete Mathematik,\\ 
Technische Universit\"at Graz,\\
Steyrergasse 30, A-8010 Graz, Austria}}
\email{woess@tugraz.at}

\date{\today\ (small correction)} 
\begin{abstract}
This note contains old instead of new results about random walks on groups, which may 
serve as a supplement to  the author's monograph \cite{Wbook}. First, we 
exhibit a basic exercise on the periodicity classes of random walk. The second topic concerns some
basics on ratio limits for random walks, which had been published ``only'' in German in the 1970ies.
\end{abstract}

\maketitle
       
\markboth{{\sf W. Woess}}
{{\sf Some old and basic facts about random walks on groups}}
\baselineskip 15pt

\section{Introduction}\label{sec:intro}

In all that follows, $\Gamma$ will be a countable, discrete group written multiplicatively,
with elements typically denoted by $x,y$, etc.,  unit element $e$, and $\mu$ a probability
measure on $\Gamma$. The support of $\mu$ is
$$
S_{\mu} = \{ x \in \Gamma : \mu(x) > 0 \}\,.
$$
Recall that the random walk on $\Gamma$ with law $\mu$ is the time-homogeneous
Markov chain with state space $\Gamma$ and transition probabilities $p_{\mu}(x,y) = \mu(x^{-1}y)$.
The $n$-step transition probabilities are then $p_{\mu}^{(n)}(x,y) = \mu^{(n)}(x^{-1}y)$,
where $\mu^{(n)}$ is the $n^{\textrm{th}}$ convolution power of $\mu$. Its support is
$S_{\mu}^n$, the set of all products $x_1 \cdots x_n$ of group elements. 
Thus, the random walk is \emph{irreducible} in the sense of Markov chains (resp., non-negative
matrics) if and only if
\begin{equation}\label{eq:semigroup}
\bigcup_{n=1}^{\infty} S_{\mu}^n = \Gamma\,,
\end{equation}
i.e., the semigroup genereated by the support is all of $\Gamma$. (In general, for
two subsets $A, B$ of $\Gamma$, their product is $AB = \{xy : x \in A\,,\; y \in B\}$,
and $A^{-1} = \{x^{-1} : x \in A \}$.)

\smallskip

$\bullet\;$ In this note, we always assume irreducibility.

\smallskip

In  \S \ref{sec:period}, we recall a basic fact about the period and aperiodicity
of the random walk, and in \S \ref{sec:specrad}, we recall some ratio limit theorems.

\section{Aperiodicity}\label{sec:period}

As an irreducible Markov chain, the random walk has a \emph{period}
$$
d = d(\mu) = \gcd\{ n \in \N : \mu^{(n)}(e) > 0 \},
$$
compare with \cite[page 3]{Wbook}. The random walk is called \emph{aperiodic,}
if $d(\mu)=1$. For general $d$, as for any irreducible Markov chain, there is a 
partition in subsets
$$
\Gamma = C_0 \,\dot{\cup}\, C_1 \,\dot{\cup}\,\dots\,\dot{\cup}\, C_{d-1}\,,
$$
with $e \in C_0\,$, such that the random walk wanders through the seta $C_j$ cyclically:
if $x \in C_j$ and $y \in C_k$ then $\mu^{(n)}(x^{-1y}) > 0$ only when 
$n \equiv k-j (\mod d)$, and then this holds fal all but finitely many such $n$. 
We now recall a group-theoretical fact which was considered a basic exercise in earlier
days but now is sometimes being ``discovered'' by younger researchers. The following
has also has a generalisation to random walks on locally compact groups, which was
studied by {\sc Woess}~\cite{Wperiod} -- which however is written in French and not available
online, apart from the author's webpage.

\begin{pro}\label{pro:normal} \hspace*{3.1cm}  $\displaystyle
\Gamma_0 = \bigcup_{k=1}^{\infty} S_{\mu}^{-k}S_{\mu}^{k}$\\[3pt]
is a normal subgroup of $\Gamma$ with index $d$. The factor group 
$\Gamma/\Gamma_0$ is cyclic of order $d$.
One has $C_0 = \Gamma_0\,$, and the sets $C_j$ are the cosets of $\Gamma_0\,$.

The probability measure $\mu^{(d)}$ is supported by $\Gamma_0$ and 
irreducible on that subgroup, with period $1$.
\end{pro}

\begin{proof}
We first show that by \eqref{eq:semigroup}, one has that $\Gamma_0$
is a normal subgroup; this can be found, e.g., in the lecture notes by 
{\sc Mukherjea and Tserpes}~\cite[p. 97]{MuTs}.
\\
(a) It is clear that $e \in \Gamma_0$ and that $x^{-1}\in\Gamma_0$ whenever $x \in \Gamma_0\,$.
\\
(b) If $x \in \Gamma$ then there is $n$ such that $x \in S_{\mu}^n$. Hence 
$x^{-1}S_{\mu}^{-k}S_{\mu}^{k}x \subset S_{\mu}^{-n-k}S_{\mu}^{k+n}$, so that
$x^{-1}\Gamma_0x \subset \Gamma_0\,$.
\\
(c) Finally, by (b), 
$$
(S_{\mu}^{-k}S_{\mu}^{k})(S_{\mu}^{-n}S_{\mu}^{n}) \subset (\Gamma_0 S_{\mu}^{-n})S_{\mu}^{n} 
= S_{\mu}^{-n}\Gamma_0 S_{\mu}^{n} \subset \Gamma_0\,,
$$
so that $\Gamma_0$ is closed with respect to the group product.

Next, we observe that $d(\mu) = \gcd N_{\mu}$, where $N_{\mu} = \{ n \in \N : e \in S_{\mu}^n \}$, 
and that the latter set
is closed under addition. This implies via elementary number theory that there are $n_1, n_2$
such that $e \in S_{\mu}^{n_1} \cap S_{\mu}^{n_2}$ and $d(\mu)=n_2-n_1\,$. Therefore
$$
e \in S_{\mu}^{-n_1} S_{\mu}^{n_1}S_{\mu}^{d} \subset \Gamma_0 S_{\mu}^{d}\,,%
\AND S_{\mu}^{-d} \subset S_{\mu}^{-d} \Gamma_0 S_{\mu}^{d} \subset \Gamma_0\,.
$$
Therefore also $S_{\mu}^{nd} \subset \Gamma_0$ for all $n \in \N$.

Now suppose that $S_{\mu}^k \cap S_{\mu}^m \ne \emptyset$ for $k \ne m$, and let $x$ be
an element of that set. By \eqref{eq:semigroup}, there is $n$ such that $x^{-1} \in S_{\mu}^n$.
Then $e \in  S_{\mu}^{k+n} \cap S_{\mu}^{m+n}$, whence $k+n$ and $m+n$ belong to $N_{\mu}\,$.
We conclude that $d$ divides $m-k$. 

In particular, let $x_0 \in S_{\mu}$. If $x_0^k \in  \Gamma_0$
then $x_0^k \in S_{\mu}^{-n} S_{\mu}^{-n}$ for some $n$, 
so that $S_{\mu}^{k+n} \cap S_{\mu}^{n} \ne \emptyset$ and $d$ divedes $k$.
Thus, the cosets $x_0^k\Gamma_0$, $k=0\,,\dots, d-1$ are all disctinct, and $x_0^d\Gamma_0 = \Gamma_0$.
This holds for any $x_0 \in S_{\mu}\,$, and it shows that $\Gamma/\Gamma_0$ is cyclic of order $d$.

Turning to the random walk, if $x \in x_0^k \Gamma^0$ for $k \in \{0,\dots, d-1\}$ 
and $p(x,y) = \mu(x^{-1}y)> 0$ then 
$y \in S_{\mu}^{k+1}$ and $y x_0^{d-1-k} \in S_{\mu}^d \subset \Gamma_0\,$. Consequently,
$y \in x_0^{k+1}\Gamma_0\,$.

By all that we have proved so far, we have that $S_{\mu}^n \cap \Gamma_0 = \emptyset$
when $d$ does not divide $n$, so that 
$$
\bigcup_{n=1}^{\infty} S_{\mu}^{nd} = \Gamma_0\,,
$$
which means that $\mu^{(d)}$ is irreducible on $\Gamma_0\,$. Since $N_{\mu}$ is closed
under addition, there is $n_0$ such that $nd \in \N_{\mu}$ for all $n \ge n_0\,$.
Therefore $\mu^{(d)}$ has period 1 on $\Gamma_0\,$.
\end{proof}

If $\mu$ has period 1, then it is called aperiodic. By the Proposition, we may
restrict attention to that case.

\section{Spectral radius and convergence}\label{sec:specrad}

For irreducible $\mu$ as in the preceding section, the number
$$
\rho(\mu) = \limsup_{n \to \infty} \mu^{(n)}(x)^{1/n}
$$
is independent of $x \in \Gamma$. It is called the \emph{spectral radius} in a variety
of references, including \cite{Wbook}. We stick to this terminology, but 
one should be careful: in general, this is not the spectral radius of an operator;
it \emph{is} the spectral radius (and norm) of $\mu$ as a convolution operator on $\ell^2(\Gamma)$
when $\mu$ is symmetric, i.e., $\mu(x^{-1}) = \mu(x)$ for all $x$, and on a 
weighted $\ell^2$-space when $\mu$ is reversible, see \cite[\S 10]{Wbook}.
The study of this number, for general irreducible Markov chains, goes back to 
{\sc Pruitt}~\cite{Pruitt} and {\sc Vere-Jones}~\cite{VJ}; see also the monograph by
{\sc Seneta}~\cite{Sen}.


\begin{lem}\label{lem:roots} For every $x \in \Gamma$, one has convergence:
$$
\lim_{n \to \infty} \mu^{(n)}(x)^{1/n} = \rho(\mu)\,,
$$
and the sequence $\mu^{(n)}(e)^{1/n}$ converges to its limit from below. Furthermore,
there is $k_0$ such that 
$$
\mu^{(n)}(e) < \rho(\mu)^n \quad \text{strictly for all }\;n \ge k_0\,. 
$$
\end{lem}

\begin{proof} The convergence result is the multiplicative version of {\sc Fekete}'s Lemma \cite{Fek}. 
Let $a_n = \mu^{(n)}(e)$. Then $a_n > 0$ for all $n \ge n_0$ and $a_m a_n \le a_{m+n} \le 1$. 
For fixed $m \in n$, let $n = qm + r$ with $n_0 \le r=r_n < n_0+m$. Then 
$$
a_n \ge a_r a_m^q  \ge c_m a_m^q \,, \quad\text{where}\quad c_m =\min \{ a_r: n_0 \le r < n_0+m\} > 0.
$$
Therefore $a_n^{1/n} \ge c_m^{1/n} a_m^{q/n}$, and letting $n \to \infty$, 
$$
\rho(\mu) \ge \liminf_{n \to \infty} a_n^{1/n} \ge a_m^{1/m}\,. 
$$
This holds for every $m$, and letting $m \to \infty$, we get that $a_m^{1/m}$ tends
to $\rho(\mu)$ from below. For general $x$, there is $k$ such that $\mu^{(k)}(x) > 0$,
and $\mu^{(n)}(x) \ge \mu^{(k)}(x)\mu^{(n-k)}(e)$. Taking $n^{\text{th}}$ roots on both
sides, we get that also $\mu^{(n)}(x)^{1/n}$ converges.

The last observation is due to {\sc Gerl}~\cite{Gerl78}: fix $x_0 \in S_{\mu} \setminus \{e\}$.
Then $x_0^{-1} \in S_{\mu}^{r_0}$ for some $r_0 \in \N$, and by the above,
$$
\mu^{(n)}(x_0)\mu^{(n)}(x_0^{-1}) > 0 \quad \text{for all }\; n \ge n_0+r_0 =:k_0
$$
Now suppose that for some $n \ge k_0$,
$\mu^{(n)}(e)=\rho(\mu)^n.$ Then 
$$
\rho(\mu)^{2n} \ge \mu^{(2n)}(e) = \sum_{x \in \Gamma} \mu^{(n)}(x)\mu^{(n)}(x^{-1})
= \rho(\mu)^{2n} + \sum_{x\ne e} \mu^{(n)}(x)\mu^{(n)}(x^{-1}).
$$
But then $\mu^{(n)}(x_0)\mu^{(n)}(x_0^{-1}) = 0$, a contradiction.
\end{proof}

The following result of {\sc Gerl}~\cite{Gerl78} is less straightforward, and was the
primary motivation for writing this little note. It is based on an argument in \cite{Gerl73},
compare also with {\sc Guivarc'h}~\cite[pp. 18-19]{Gui}. 
According to {\sc Le Page}~\cite{LeP}, the argument can be traced back to 
{\sc Orey and Kingman}~\cite{OK}.

\begin{theorem1}\label{thm:ratio} If $\mu$ is irreducible and aperiodic on $\Gamma$,
then
$$
\lim_{n \to \infty} \frac{\mu^{(n+1)}(x)}{\mu^{(n)}(x)} = \rho(\mu) 
\quad \text{for every }\; x \in \Gamma\,.
$$ 
\end{theorem1}

\begin{proof} Write $\rho = \rho(\mu)$. 
Recall the transition probabilities $p(x,y) = \mu(x^{-1}y)$ of the random walk.
Irreducibility yields that there is a positive $\rho$-subharmonic function 
$h: \Gamma \to (0\,,\,\infty)$, that is,
$$
\sum_{y \in \Gamma} \mu(x^{-1}y)h(y) \le \rho\, h(x) \quad \text{for every }\; x \in \Gamma\,.
$$
See \cite{Pruitt}, \cite{Sen} or \cite[Lemma 7.2]{Wbook}. In many cases, there even is
such a function which is $\rho$-harmonic, i.e., equality holds at every $x$.

We now define a new Markov chain on $\Gamma$, the \emph{$h$-process,} with transition probabilities
$$
p_h(x,y) = \frac{\mu(x^{-1}y)h(y)}{\rho\,h(x)}.
$$
This is in general not a group-invariant random walk. The
transition matrix (denoted $Q$ in \cite{Gerl78})
$$
P_h = \bigl(p_h(x,y)\bigr)_{x,y \in \Gamma}
$$
is substochastic, i.e., $\sum_y p_h(x,y) \le 1$ for all $x$, so that there may be a positive
probability that the Markov chain ``dies'' at $x$. Furthermore, along with $\mu$ it is 
irreducible and aperiodic: for all $x,y$, there is $n_{x,y}$ such that $p_h^{(n)}(x,y) > 0$
for all $n \ge n_{x,y}\,$, where $p_h^{(n)}(x,y)$
is the $(x,y)$-entry of the matrix power $P_h^n\,$. We have by Lemma \ref{lem:roots}
$$
p_h^{(n)}(x,y) = \frac{\mu^{(n)}(x^{-1}y)h(y)}{\rho^n\,h(x)} \AND 
\lim_{n\to \infty} p_h^{(n)}(x,y)^{1/n} = 1 \quad \text{for all }\;x,y \in \Gamma\,.
$$
In particular, for all $x$,
$$
0 < p_h^{(n)}(x,x) = \frac{\mu^{(n)}(e)}{\rho^n} < 1 \quad \text{for all $n \ge k_0\,$.}
$$
We now fix $m \ge k_0$ and set $a = a_m = 1- p_h^{(m)}(x,x)$, so that $0 < a < 1$, as well as
$$
Q = \frac{1}{a} \bigl( P_h^m - (1-a)I\bigr)\,,
$$
where $I$ is the identity matrix over $\Gamma$.  (The matrix $Q$ is denoted $R$ in \cite{Gerl78}.)
We shall also need the matrix $E$ over $\Gamma$ with all entries $=1$.
For the next lines of the proof, we just write $P$ for $P_h^m\,$.
Then $Q$ is also non-negative, substochastic and irreducible, and $P = a\, Q + (1-a)I$. Note 
that $Q$ commutes with $P_h\,$, and that $P_hE \le E$. Then
$$
P^n = \sum_{k=0}^n p_a(n,k) \,Q^k\,, \quad\text{where}\quad  
p_a(n,k) = {n \choose k}a^k(1-a)^{n-k}\,.
$$
The latter is the probability that the sum $S_n = X_1 +\dots+X_n$ of i.i.d. Bernoulli
random variables with $\Prob[X_k=1] = a$ has value $k$. For $\ep > 0$, consider the set
$$
C_n(\ep) = \bigl\{ k \in \{0\,,\dots, n\} : 
p_a(n,k) \le (1+\ep)\,p_a(n+1,k+1) \bigr\}
$$
and its complement $C_n(\ep)^c$ in  $\{0\,,\dots, n\}$. Then
$$
\sum_{k \in C_n^c} p_a(n,k) = \Prob \biggl[   \frac{S_n+1}{n+1}-a > \ep\,a \biggr]\,.
$$
This is a large deviation probability, which is well known to decay exponentially, 
i.e., there is $\delta > 0$
such that it is $\le e^{-n\delta}$. In our specific case, this can also be verified
by combinatorial computations. See e.g. {\sc R\'enyi}~\cite[p. 324]{Renyi}.
Then, using matrix products and elementwise inquality between matrices,
$$
\begin{aligned}
\frac{1}{a}P^{n+1}  - \frac{1-a}{a}P^n = QP^n &= \sum_{k=0}^n p_a(n,k)Q^{k+1}\\
&\le e^{-\delta n}E + (1+\ep)\sum_{k \in C_n(\ep)} p_a(n+1,k+1)Q^{k+1}\\ 
&\le e^{-\delta n}E + (1+\ep)P^{n+1}\,.
\end{aligned}
$$
Reassembling the terms,
$$
\Bigl( 1 - \frac{a}{1-a}\, \ep\Bigr)P^{n+1} \le \frac{a}{1-a}\,e^{-\delta n}E + P^n\,.
$$
We multiply from the left with $P_h^r$, where $r \in \N_0$ is arbitrary, and
get for the matrix elements
$$
\left( 1 - \frac{a}{1-a} \ep\right)\frac{p_h^{(mn + m + r)}(x,y)}{p_h^{(mn + r)}(x,y)} \le 
 \frac{a}{1-a}\,\frac{e^{-\delta n}}{p_h^{(mn + r)}(x,y)} + 1\,.
$$
We are not dividing by $0$ if $n$ is sufficiently large, and since 
$p_h^{(mn + r)}(x,y)^{1/n} \to 1$, the right hand side tends to 1 as $n \to \infty$.
Since we can choose $\ep$ arbitrarily small, we get
\begin{equation}\label{eq:limsup}
\limsup_{n\to \infty} \frac{p_h^{(mn + m + r)}(x,y)}{p_h^{(mn + r)}(x,y)} \le 1\,,
\end{equation}
and this holds for every $m \ge k_0$ and every $r \ge 0$.

\smallskip

For an analogous lower bound, we use the set 
$$
D_n(\ep) = \bigl\{ k \in \{0\,,\dots, n\} : 
p_a(n+1,k+1) \le (1+\ep)p_a(n,k) \bigr\}
$$
and observe that
$$
\sum_{k \in C_n^c} p_a(n+1,k+1) = \Prob \biggl[\frac{S_{n+1}+1}{n+1}-a < - \frac{\ep}{1+\ep}\,a \biggr]\,.
$$
also decays exponentially, and is bounded by $e^{-\delta n}$ for some $\delta > 0$. Then
$$
\begin{aligned}
P^{n+1}  &\le e^{-\delta n}E + (1+\ep)\sum_{k \in D_n(\ep)} p_a(n,k)Q^{k+1}\\
&\le e^{-\delta n}E + (1+\ep)P^n Q 
= e^{-\delta n}E + (1+\ep)\left(\frac{1}{a}P^{n+1}  - \frac{1-a}{a}P^n\right)\,.
\end{aligned}
$$
Reassembling the terms,
$$
(1+\ep)P^n \le \frac{a}{1-a}\,e^{-\delta n}E + \Bigl(1 + \frac{\ep}{1-a}\Bigr)P^{n+1}\,.
$$
Proceeding as above, we get for all $m \ge k_0$ and all $r \ge 0$
\begin{equation}\label{eq:liminf}
\liminf_{n\to \infty} \frac{p_h^{(mn + m + r)}(x,y)}{p_h^{(mn + r)}(x,y)} \ge 1\,,
\end{equation}
for every $m \ge k_0$ and every $r \ge 0$. Since these two numbers chan be chosen arbitrarily,
it is an easy exercise to deduce from \eqref{eq:limsup} and \eqref{eq:liminf} that 
$$
\lim_{n\to \infty} \frac{p_h^{(n+1)}(x,y)}{p_h^{(n)}(x,y)} = 1\,,
$$
and the stated result follows.
\end{proof}

\begin{rmk}\label{rmk:nogroups}
The last theorem is not restricted to random walks on groups. If $P$ is the transition matrix
of an irreducible Markov chain on a countable set -- say -- $\Gamma$, and it is
\emph{strongly aperiodic,} that is,
$$
\gcd \bigl\{ n \in \N : \inf_x p^{(n)}(x,x) > 0 \bigr\} = 1\,,
$$
then the same proof applies to show that
$$
\lim_{n\to \infty} \frac{p^{(n+1)}(x,y)}{p^{(n)}(x,y)} = \rho(P) \quad \text{for all }\; 
x,y \in \Gamma\,. \eqno\square
$$
\end{rmk}

In the situation of Theorem 1, assume for simplicity that $S_{\mu}$ is finite. 
Then it is well known and easy to decuce that the sequence of measures
$$
\Bigl(\frac{\mu^{(n)}}{\mu^{(n)}(e)}\Bigr)_{n \in \N}
$$
is relatively compact in the topology of pointwise convergence, and every 
limit measure $\nu$ satisfies the convolution equation
\begin{equation}\label{eq:conv}
 \mu * \nu = \rho(\mu)\cdot \nu\,,
\end{equation}
or in other words, the function $h(x) = \nu(x^{-1})$
is $\rho(\mu)$-harmonic,
that is 
$$
\sum_y \mu(x^{-1}y)h(y) = \rho(\mu)\, h(x) \quad \text{for all }\; x \in \Gamma\,. 
$$
\cite{Gerl78} and  (under slightly different conditions, where
the state space is not necessarily discrete) \cite{Gui}  prove the following ratio limit theorem.

\begin{theorem2}\label{thm:2}
Assume that $\mu$ is irreducible and aperiodic, and that $S_{\mu}$ is finite.
Suppose that {\sf (P)} is a certain property of positive measures $\nu$ on $\Gamma$
such that
\begin{itemize}
 \item every limit along a pointwise convergent subsequence 
$$
\Bigl(\frac{\mu^{(n_k)}}{\mu^{(n_k)}(e)}\Bigr)_{k \in \N}
$$
must have property {\sf (P)}, and
\item there is a unique positive measure $\nu$ with property {\sf (P)} that satisfies
$\nu(e)=1$ and~\eqref{eq:conv}.
\end{itemize}
Then 
$$
\lim_{n \to \infty}\frac{\mu^{(n)}(x)}{\mu^{(n)}(e)} = \nu(x) \quad \text{for all }\; x \in \Gamma.
$$
\end{theorem2}

The proof is clear in view of the observations preceding the statement of Theorem 2, 
i.e., relative compactness and \eqref{eq:conv}, which are left to the reader as exercises
or to be looked up in old references. Finiteness of $S_{\mu}$ can be relaxed.
Typical examples of application are (virtually) Abelian groups, where property {\sf (P)}
is empty, because there is a unique solution $\nu$ of \eqref{eq:conv}: there is a good 
amount of literature from the 1960ies on this. The most significant reference for Abelian
groups is {\sc Stone}~\cite{Stone}.

Other typical examples are isotropic random walks on free groups, and property {\sf (P)}
is that $\nu$ is also isotropic; see \cite{Gerl78}.
It should be noted that in those cases, one also has stronger results, namely 
local limit theorems; see \cite[Chapter III]{Wbook}. 

\smallskip

\textbf{\, Author's final, personal remarks.} When I wrote the monograph \cite{Wbook} in the 
$2^{\text{nd}}$ half of the 1990ies, ratio limit theorems were not an active topic,
but replaced by the study of the asymptotic type of random walk transition probabilities and
the sharper local limit theorems. Therefore, having to keep the book size under control,
I had ``sacrificed'' the material of ratio limit theorems. In the meantime, the 
subject has ``woken up'' again, e.g. in a recent paper of {\sc Dor-On}~\cite{Doron} and current
work of {\sc Dougall and Sharp}.

\end{document}